\documentclass[12pt]{article}%
\usepackage[intlimits]{amsmath}
\usepackage{amssymb}
\usepackage{latexsym}
\usepackage{tikz}
\usepackage{anysize}
\usepackage{subfigure}
\usepackage[T1]{fontenc}
\usepackage[sc]{mathpazo}
\usepackage{color}
\usepackage[colorlinks]{hyperref}
\usepackage{amsfonts}
\usepackage{microtype}
\usepackage{graphicx}%
\definecolor {refcol}{RGB}{40,0,255}
\hypersetup{colorlinks=true,allcolors=refcol}
\makeatletter
\newfont{\footsc}{cmcsc10 at 8truept}
\newfont{\footbf}{cmbx10 at 8truept}
\newfont{\footrm}{cmr10 at 10truept}
\newtheorem{theorem}{Theorem}

\newtheorem{conjecture}[theorem]{Conjecture}
\newtheorem{corollary}[theorem]{Corollary}

\newtheorem{lemma}[theorem]{Lemma}

\newtheorem{proposition}[theorem]{Proposition}
\newtheorem{question}[theorem]{Question}

\newenvironment{proof}[1][Proof]{\noindent{\textbf {#1}  }}  {\hfill$\Box$\bigskip}
\begin{document}

\title{\textbf{On the }$\alpha$\textbf{-index of graphs with pendent paths}}
\author{Vladimir Nikiforov\thanks{Department of Mathematical Sciences, University of
Memphis, Memphis TN 38152, USA.} \ and Oscar Rojo\thanks{Department of
Mathematics, Universidad Cat\'{o}lica del Norte, Antofagasta, Chile.}}
\date{}
\maketitle

\begin{abstract}
Let $G$ be a graph with adjacency matrix $A(G)$ and let $D(G)$ be the diagonal
matrix of the degrees of $G$. For every real $\alpha\in\left[  0,1\right]  $,
write $A_{\alpha}\left(  G\right)  $ for the matrix
\[
A_{\alpha}\left(  G\right)  =\alpha D\left(  G\right)  +(1-\alpha)A\left(
G\right)  .
\]
This paper presents some extremal results about the spectral radius
$\rho_{\alpha}\left(  G\right)  $ of $A_{\alpha}\left(  G\right)  $ that
generalize previous results about $\rho_{0}\left(  G\right)  $ and $\rho
_{1/2}\left(  G\right)  $.

In particular, write $B_{p,q,r}$ be the graph obtained from a complete graph
$K_{p}$ by deleting an edge and attaching paths $P_{q}$ and $P_{r}$ to its
ends. It is shown that if $\alpha\in\left[  0,1\right)  $ and $G$ is a graph
of order $n$ and diameter at least $k,$ then%
\[
\rho_{\alpha}(G)\leq\rho_{\alpha}(B_{n-k+2,\lfloor k/2\rfloor,\lceil
k/2\rceil}),
\]
with equality holding if and only if $G=B_{n-k+2,\lfloor k/2\rfloor,\lceil
k/2\rceil}$. This result generalizes results of Hansen and Stevanovi\'{c}
\cite{HaSt08}, and Liu and Lu \cite{LiLu14}.

\end{abstract}

T\textbf{AMS classification: }\textit{ 05C50, 15A48}


\section{Introduction}

Let $G$ be a graph with adjacency matrix $A(G)$, and let $D\left(  G\right)  $
be the diagonal matrix of its vertex degrees. In \cite{Nik17} the matrix
$A_{\alpha}(G)$ has been defined for any real $\alpha\in\left[  0,1\right]  $
as
\[
A_{\alpha}(G)=\alpha D(G)+(1-\alpha)A(G).
\]
Write $Q\left(  G\right)  $ for the signless Laplacian $A(G)+D(G)$ of $G$ and
note that $A_{0}\left(  G\right)  =A\left(  G\right)  $ \ \ and
\ \ \ $2A_{1/2}\left(  G\right)  =Q\left(  G\right)  $; thus, the family
$A_{\alpha}(G)$ extends both $A\left(  G\right)  $ and $Q\left(  G\right)  $.

Write $\rho_{a}\left(  G\right)  $ for the spectral radius of $A_{\alpha}(G)$
and call $\rho_{\alpha}\left(  G\right)  $ the $\alpha$\emph{-index} of $G$.
In the spirit of the general problem of Brualdi and Solheid \cite{BrSo86}, one
can ask how large or how small can be the $\alpha$-index of graphs with some
specific properties. For example: how large $\rho_{\alpha}\left(  G\right)  $
can be if $G$ is graph of order $n$ and diameter at least $k$? In fact, for
$\alpha=0$ this question has been answered by Hansen and Stevanovi\'{c} in
\cite{HaSt08}.

Denote by $B_{p,q,r}$ the graph obtained from a complete graph $K_{p}$ by
deleting an edge and attaching paths $P_{q}$ and $P_{r}$ to its
ends\footnote{Hansen and Stevanovi\'{c} call these graphs \emph{bugs.
}Recently, using some advanced techniques, the spectrum of $B_{p,q,r}$ has
been calculated in \cite{Roj17}.} (see Fig. 1 for an example). With this
definition, Hansen and Stevanovi\'{c}'s result reads as:

\begin{theorem}
[\textbf{Hansen, Stevanovi\'{c} \cite{HaSt08}}]\label{ad}Let $G$ be a graph of
order $n$ with $\mathrm{diam}\left(  G\right)  \geq k$. If $k=1$, then
$\rho_{0}(G)=\rho_{0}(K_{n})$. If $k\geq2$, then
\[
\rho_{0}\left(  G\right)  \leq\rho_{0}(B_{n-k+2,\lfloor k/2\rfloor,\lceil
k/2\rceil}),
\]
with equality holding if and only if $G=B_{n-k+2,\lfloor k/2\rfloor,\lceil
k/2\rceil}$.
\end{theorem}

More recently, Liu and Lu \cite{LiLu14} proved the same result for the
spectral radius of $Q\left(  G\right)  $, that is, for $\alpha=1/2$:

\begin{theorem}
[\textbf{Liu,Lu \cite{LiLu14}}]Let $G$ be a graph of order $n$ with
$\mathrm{diam}\left(  G\right)  \geq k$. If $k=1$, then $\rho_{1/2}\left(
G\right)  =\rho_{1/2}(K_{n})$. If $k\geq2$, then
\[
\rho_{1/2}\left(  G\right)  \leq\rho_{1/2}(B_{n-k+2,\lfloor k/2\rfloor,\lceil
k/2\rceil}),
\]
with equality if and only if $G=B_{n-k+2,\lfloor k/2\rfloor,\lceil k/2\rceil}$.
\end{theorem}

\begin{figure}[ptb]%
\begin{align*}
\begin{tikzpicture} \tikzstyle{every node}=[draw,circle,fill=black,minimum size=4pt, inner sep=0pt] \draw (-4,-1) node (1) [label=above:$$] {} (-3,-1) node (2) [label=left:$$] {} (-2,-1) node (3) [label=below:$$] {} (1,-1) node (4) [label=below:$$] {} (2,-1) node (5) [label=right:$$] {} (3,-1) node (6) [label=right:$$] {} (4,-1) node (7) [label=below:$$] {} (5,-1) node (8) [label=right:$$] {} (-2,0) node (9) [label=below:$$] {} (-1,1) node (10) [label=below:$$] {} (0,1) node (11) [label=above:$$] {} (1,0) node (12) [label=right:$$] {}; \draw (1)--(2); \draw (2)--(3); \draw (4)--(5); \draw (6)--(5); \draw (7)--(6); \draw (8)--(7); \draw (3)--(9); \draw (3)--(10); \draw (3)--(11); \draw (3)--(12); \draw (4)--(9); \draw (4)--(10); \draw (4)--(11); \draw (4)--(12); \draw (9)--(10); \draw (9)--(11); \draw (9)--(12); \draw (10)--(11);\draw (10)--(12); \draw (11)--(12); \end{tikzpicture}
\end{align*}
\caption{The graph $B_{6,3,5}$}%
\end{figure}
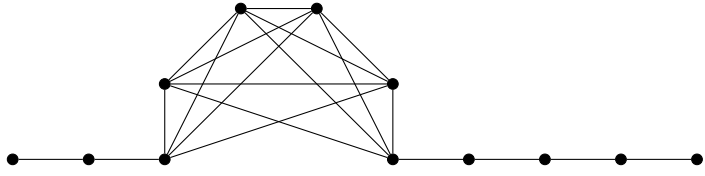

One of the main goals of this paper is to extend the above theorems for all
$\alpha\in\left[  0,1\right)  $:

\begin{theorem}
\label{mth}Let $\alpha\in\left[  0,1\right)  $ and $G$ be a graph of order $n$
with $\mathrm{diam}\left(  G\right)  \geq k$. If $k=1$, then $\rho_{\alpha
}(G)=\rho_{\alpha}(K_{n})$. If $k\geq2$, then
\[
\rho_{\alpha}(G)\leq\rho_{\alpha}(B_{n-k+2,\lfloor k/2\rfloor,\lceil
k/2\rceil}),
\]
with equality if and only if $G=B_{n-k+2,\lfloor k/2\rfloor,\lceil k/2\rceil}$.
\end{theorem}

A related extremal result is about connected graphs with given clique number.
Denote by $PK_{p,q}$ the graph obtained by joining an end-vertex of the path
$P_{p}$ to a vertex of the complete graph $K_{q}$. In \cite{StHa08},
Stevanovi\'{c} and Hansen proved the following theorem:

\begin{theorem}
\label{thHS}If $G$ is a connected graph of order $n$ with clique number
$\omega\geq2$, then%
\[
\rho_{0}\left(  G\right)  \geq\rho_{0}\left(  PK_{n-\omega,\omega}\right)  ,
\]
with equality if and only if $G=PK_{n-\omega,\omega}$.
\end{theorem}

A different (and more involved) proof of Theorem \ref{thHS} was given in
\cite{ZHG14} by Zhang, Huang, and Guo, who, in fact, determined the largest
four values of $\rho_{0}\left(  G\right)  $.

We generalize Theorem \ref{thHS} for any $\alpha\in\left[  0,1\right)  $:

\begin{theorem}
\label{thcl}Let $\alpha\in\left[  0,1\right)  $. If $G$ is a connected graph
of order $n$ with clique number $\omega\geq2$, then%
\[
\rho_{\alpha}\left(  G\right)  \geq\rho_{\alpha}\left(  PK_{n-\omega,\omega
}\right)  ,
\]
with equality if and only if $G=PK_{n-\omega,\omega}$
\end{theorem}

The rest of the paper is organized as follows. In Section \ref{spp} we study
the distribution of the entries of Perron vectors of the $\alpha$-index along
pendent paths in graphs. These results are used in Section \ref{smth} and in
Section \ref{sthcl} to carry out the proofs of Theorem \ref{mth} and Theorem
\ref{thcl}. In the last Section \ref{ops} we raise some open problems inspired
by results of Li and Feng \cite{LiFe79}, whose solution could provide new
tools in the study of the $\alpha$-index.

\section{\label{snot}Notation and preliminaries}

Given a graph $G$ and a vertex $u\in V\left(  G\right)  $, we write
$\Gamma_{G}\left(  u\right)  $ for the set of neighbors of $u$ and set
$d_{G}\left(  u\right)  :=\left\vert \Gamma_{G}\left(  u\right)  \right\vert
$. As usually, $P_{n}$ and $K_{n}$ denote the path and the complete graph of
order $n$, and $K_{n}-e$ stands for $K_{n}$ with an edge removed.

For an $n\times n$ symmetric matrix $A=\left[  a_{i,j}\right]  $ and a vector
$\mathbf{x}:=\left(  x_{1},\ldots,x_{n}\right)  $, we write $\left\langle
A\mathbf{x},\mathbf{x}\right\rangle $ for the quadratic form of $A$, i.e.,%
\[
\left\langle A\mathbf{x},\mathbf{x}\right\rangle =\sum_{i,j}a_{i,j}x_{i}%
x_{j}.
\]

In our proofs, we frequently use the following lemma that generalizes results
known for the adjacency matrix and the signless Laplacian of graphs.

\begin{lemma}
\label{rol}Let $\alpha\in\left[  0,1\right)  $ and let $G$ be a graph of order
$n$. Suppose that $u,v\in V(G)$ and $S\subset V(G)$ satisfy $u,v\notin S$ and
for every $w\in S$, $\{u,w\}\in E(G)$ and $\{v,w\}\notin E\left(  G\right)  $.
Let $H$ be the graph obtained by deleting the edges $\{u,w\}$ and adding the
edges $\{v,w\}$ for all $w\in S$. If $S$ is nonempty and there is a positive
eigenvector $\left(  x_{1},\ldots,x_{n}\right)  $ to $\rho_{\alpha}\left(
G\right)  $ such that $x_{v}\geq x_{u}$, then
\[
\rho_{\alpha}\left(  H\right)  >\rho_{\alpha}\left(  G\right)  .
\]

\end{lemma}

\begin{proof}
Calculating the quadratic forms $\left\langle A_{\alpha}\left(  H\right)
\mathbf{x},\mathbf{x}\right\rangle $ and $\left\langle A_{\alpha}\left(
G\right)  \mathbf{x},\mathbf{x}\right\rangle $, we see that%
\begin{align*}
\rho_{\alpha}\left(  H\right)  -\rho_{\alpha}\left(  G\right)   &
\geq\left\langle A_{\alpha}\left(  H\right)  \mathbf{x},\mathbf{x}%
\right\rangle -\left\langle A_{\alpha}\left(  G\right)  \mathbf{x}%
,\mathbf{x}\right\rangle \\
&  =\alpha\left(  \left\langle D\left(  H\right)  \mathbf{x},\mathbf{x}%
\right\rangle -\left\langle D\left(  G\right)  \mathbf{x},\mathbf{x}%
\right\rangle \right)  +\left(  1-\alpha\right)  \left(  \left\langle A\left(
H\right)  \mathbf{x},\mathbf{x}\right\rangle -\left\langle A\left(  G\right)
\mathbf{x},\mathbf{x}\right\rangle \right) \\
&  =-\alpha\left\vert S\right\vert x_{u}^{2}+\left\vert S\right\vert \alpha
x_{v}^{2}+2\left(  1-\alpha\right)  \sum_{w\in S}x_{w}\left(  -x_{u}%
+x_{v}\right) \\
&  =\left(  x_{v}-x_{u}\right)  \sum_{w\in S}\left(  \alpha x_{v}+\alpha
x_{u}+2\left(  1-\alpha\right)  x_{w}\right)  \geq0.
\end{align*}
However, the equality $\rho_{\alpha}\left(  H\right)  =\rho_{\alpha}\left(
G\right)  $ is not possible, for otherwise $\mathbf{x}$ is a positive
eigenvector to $\rho_{\alpha}\left(  H\right)  $, leading to the contradictory
equations%
\begin{align*}
\rho_{\alpha}\left(  G\right)  x_{v}  &  =\alpha d_{G}\left(  v\right)
x_{v}^{2}+\left(  1-\alpha\right)  \sum_{i\in\Gamma_{G}\left(  v\right)
}x_{i}\\
\rho_{\alpha}\left(  H\right)  x_{v}  &  =\alpha d_{H}\left(  v\right)
x_{v}^{2}+\left(  1-\alpha\right)  \sum_{i\in\Gamma_{H}\left(  v\right)
}x_{i}>\alpha d_{G}\left(  v\right)  x_{v}^{2}+\left(  1-\alpha\right)
\sum_{i\in\Gamma_{G}\left(  v\right)  }x_{i}.
\end{align*}

\end{proof}

\section{\label{spp}Graphs with pendent paths}

Let $G$ be a connected graph containing a path $P$ as a subgraph. We say that
$P$ is a \emph{pendent path} if one of its ends is a cut vertex of $G$; call
this vertex the \emph{root} of $P$. Note that a graph can have multiple
pendent paths, which may share roots; e.g., the graph $B_{p,q,r}$ has two
pendent paths. Graphs with pendent paths arise often in spectral extremal
graph theory; thus it is worth to investigate their $\alpha$-index in some generality.

For any vertex $u$ of a connected graph $G$, let $G_{p,q}\left(  u\right)  $
be the graph obtained by attaching the paths $P_{p}$ and $P_{q}$ to $u$.
Similarly, for any two vertices $u$ and $v$ of a connected graph $G$, let
$G_{p,q}\left(  u,v\right)  $ be the graph obtained by attaching the paths
$P_{p}$ to $u$ and $P_{q}$ to $v$.

Let $\alpha\in\left[  0,1\right)  $. If $G$ is a graph with a pendent path $P$
and $\rho_{\alpha}\left(  G\right)  =\rho\geq2$, then the distribution of the
entries of an eigenvector to $\rho$ along $P$ is well determined. To prove
this fact, we use the crucial equation
\begin{equation}
X^{2}-\frac{\rho-2\alpha}{1-\alpha}X+1=0. \label{ceq}%
\end{equation}
Thus, write $\gamma$ for the root of (\ref{ceq})%
\begin{equation}
\gamma:=\frac{1}{2}\left(  \frac{\rho-2\alpha}{1-\alpha}+\sqrt{\left(
\frac{\rho-2\alpha}{1-\alpha}\right)  ^{2}-4}\right)  , \label{gam}%
\end{equation}
and note that $\gamma$ is real since $\rho\geq2$; moreover, $\gamma\geq1$,
with strict inequality if $\rho>2$. Note also that the other root of
(\ref{ceq}) is equal to $\gamma^{-1}$.

Our first statement shows that the entries of an eigenvector to $\rho_{\alpha
}\left(  G\right)  $ decay exponentially along pendent paths in $G$.

\begin{proposition}
\label{pro1}Let $\alpha\in\left[  0,1\right)  $, and let $G$ be a graph with
$\rho:=\rho_{\alpha}\left(  G\right)  \geq2$. Let $P=\left(  u_{1}%
,\ldots,u_{r+1}\right)  $ be a pendent path \ in $G$ with root $u_{1}$. Let
$x_{1},\ldots,x_{r+1}$ be the entries of a positive unit eigenvector to
$\rho_{\alpha}\left(  G\right)  $ corresponding to $u_{1},\ldots,u_{r+1}$. If
$\gamma$ is defined by (\ref{gam}), then, for every $i=1,\ldots,r$, we have%
\begin{equation}
x_{i}>\gamma x_{i+1}. \label{ub1}%
\end{equation}

\end{proposition}

\begin{proof}
The eigenequation of $A_{\alpha}\left(  G\right)  $ for $x_{r+1}$ is%
\[
\rho x_{r+1}=\alpha x_{r+1}+\left(  1-\alpha\right)  x_{r},
\]
yielding in turn%
\[
x_{r}=\frac{\rho-\alpha}{1-\alpha}x_{r+1}>\frac{1}{2}\left(  \frac
{\rho-2\alpha}{1-\alpha}+\sqrt{\left(  \frac{\rho-2\alpha}{1-\alpha}\right)
^{2}-4}\right)  x_{r+1}=\gamma x_{r+1}.
\]
Proceeding by induction, the eigenequation of $A_{\alpha}\left(  G\right)  $
for $x_{i}$, together with the induction assumption, gives%
\begin{align*}
\frac{\rho-2\alpha}{1-\alpha}x_{i}  &  =x_{i+1}+x_{i-1}\leq\gamma^{-1}%
x_{i}+x_{i-1}\\
&  =\left(  \frac{\rho-2\alpha}{1-\alpha}-\frac{1}{2}\left(  \frac
{\rho-2\alpha}{1-\alpha}-\sqrt{\left(  \frac{\rho-2\alpha}{1-\alpha}\right)
^{2}-4}\right)  \right)  x_{i}+x_{i-1}.
\end{align*}
Hence,%
\[
x_{i-1}\geq\left(  \frac{\rho-2\alpha}{1-\alpha}+\frac{1}{2}\left(  \frac
{\rho-2\alpha}{1-\alpha}-\sqrt{\left(  \frac{\rho-2\alpha}{1-\alpha}\right)
^{2}-4}\right)  \right)  x_{i}=\gamma x_{i},
\]
completing the proof of Proposition \ref{pro1}.
\end{proof}

Note the following simple, but useful consequence of Proposition \ref{pro1}:

\begin{corollary}
\label{cor0}Given the hypotheses of Proposition \ref{pro1}, we have%
\[
x_{1}>\cdots>x_{r+1}.
\]

\end{corollary}

It turns out the inequality (\ref{ub1}) is quite sharp. First, using linear
recurrences, we obtain more precise information about $x_{1},\ldots,x_{r+1}$:

\begin{lemma}
\label{le1}With the hypotheses of Proposition \ref{pro1}, for every
$i=1,\ldots,r+1$, we have
\begin{equation}
x_{i}=A\gamma^{r+2-i}+B\gamma^{i-r-2}, \label{eq1}%
\end{equation}
where $\left(  A,B\right)  $ is the solution to the linear system%
\begin{align}
X+Y  &  =-\frac{\alpha}{1-\alpha}x_{r+1},\label{sys}\\
X\gamma+Y\gamma^{-1}  &  =x_{r+1}.\nonumber
\end{align}

\end{lemma}

\begin{proof}
Let $\left(  A,B\right)  $ be the (unique) solution to system (\ref{sys}). If
$r=1$, then%
\[
x_{2}=A\gamma+B\gamma^{-1},
\]
and
\begin{align*}
A\gamma^{2}+B\gamma^{-2}  &  =\frac{\rho-2\alpha}{1-\alpha}A\gamma
-A+\frac{\rho-2\alpha}{1-\alpha}B\gamma^{-1}-B\\
&  =\frac{\rho-2\alpha}{1-\alpha}x_{2}+\frac{\alpha}{1-\alpha}x_{2}=x_{1},
\end{align*}
proving the assertion for $r=1$. If $r>1$, define the sequence $z_{0}%
,\ldots,z_{r+1}$ by letting
\begin{equation}
z_{0}:=-\frac{\alpha}{1-\alpha}x_{r+1}, \label{z0}%
\end{equation}
and $z_{i}:=x_{r+2-i}$ for each $i=1,\ldots,r+1$. For $i=2,\ldots,r$ the
eigenequations of $A_{\alpha}\left(  G\right)  $ imply that
\[
\frac{\rho-2\alpha}{1-\alpha}x_{i}=x_{i+1}+x_{i-1}.
\]
Hence, $z_{1},\ldots,z_{r+1}$ satisfy the linear recurrence%
\[
z_{i+1}=\frac{\rho-2\alpha}{1-\alpha}z_{i}-z_{i-1}.
\]
for $i=2,\ldots,r$. Note that the choice of $z_{0}$ in equation (\ref{z0})
ensures that the above equality holds also for $i=1$. Hence, as known from the
theory of linear recurrences, for $i=1,\ldots,r+1$, we have
\[
z_{i}=A\gamma^{i}+B\gamma^{-i},
\]
where $A$ and $B$ are determined by the conditions%
\begin{align*}
A\gamma^{0}+B\gamma^{0}  &  =z_{0}=-\frac{\alpha}{1-\alpha}x_{r+1}\\
A\gamma+B\gamma^{-1}  &  =z_{1}=x_{r+1}.
\end{align*}
Clearly, $\left(  A,B\right)  $ is the solution to system (\ref{sys}).
Returning back to $x_{i}$, we obtain (\ref{eq1}).
\end{proof}

Now, using Lemma \ref{le1}, we give lower bounds which show that Proposition
\ref{pro1} is quite sharp:

\begin{lemma}
\label{le2}Given the hypotheses of Proposition \ref{pro1}, for every
$i=1,\ldots,r$, we have
\begin{equation}
\frac{x_{i}}{x_{1}}>\left(  1-\gamma^{-2}\right)  \gamma^{-i+1}, \label{lbg}%
\end{equation}
and
\begin{equation}
\frac{x_{r+1}}{x_{1}}>\frac{\left(  \gamma^{2}-1\right)  \left(
1-\alpha\right)  }{\gamma\left(  \left(  1-\alpha\right)  \gamma
+\alpha\right)  }\gamma^{-r}. \label{lb}%
\end{equation}

\end{lemma}

\begin{proof}
Let $\left(  A,B\right)  $ be the solution to system (\ref{sys}). Lemma
\ref{le1} implies that
\[
\frac{x_{i}}{x_{1}}=\frac{A\gamma^{r+2-i}+B\gamma^{i-r-2}}{A\gamma
^{r+1}+B\gamma^{-r-1}}%
\]
for $i=1,\ldots,r+1$. Solving system (\ref{sys}), we get
\begin{align}
A  &  =\frac{\gamma-\alpha\left(  \gamma-1\right)  }{\left(  \gamma
^{2}-1\right)  \left(  1-\alpha\right)  }x_{r+1},\label{eqa}\\
B  &  =-\frac{\gamma\left(  \alpha\left(  \gamma-1\right)  +1\right)
}{\left(  \gamma^{2}-1\right)  \left(  1-\alpha\right)  }x_{r+1}.\nonumber
\end{align}
Hence $A>0,$ $B<0,$ and
\[
\frac{B}{A}=-\frac{\gamma\left(  \alpha\left(  \gamma-1\right)  +1\right)
}{\gamma-\alpha\left(  \gamma-1\right)  }>-\frac{\gamma\left(  1\left(
\gamma-1\right)  +1\right)  }{\gamma-1\left(  \gamma-1\right)  }=-\gamma^{2}.
\]
Now, if $1\leq i\leq r$, we find that
\begin{align*}
\frac{x_{i}}{x_{1}}  &  =\frac{A\gamma^{r+2-i}+B\gamma^{i-r-2}}{A\gamma
^{r+1}+B\gamma^{-r-1}}>\frac{A\gamma^{r+2-i}+B\gamma^{i-r-2}}{A\gamma^{r+1}%
}=\gamma^{-i+1}+\frac{B}{A}\gamma^{i-2r-3}\\
&  >\gamma^{-i+1}-\gamma^{i-2r-1}=\left(  1-\gamma^{2i-2r-2}\right)
\gamma^{-i+1}\geq\left(  1-\gamma^{-2}\right)  \gamma^{-i+1},
\end{align*}
proving (\ref{lbg}).

To prove (\ref{lb}), note that (\ref{eq1}) and (\ref{eqa}) imply that
\[
\frac{x_{r+1}}{x_{1}}=\frac{x_{r+1}}{A\gamma^{r+1}+B\gamma^{-r-1}}%
>\frac{x_{r+1}}{A\gamma^{r+1}}=\frac{\left(  \gamma^{2}-1\right)  \left(
1-\alpha\right)  }{\gamma\left(  \left(  1-\alpha\right)  \gamma
+\alpha\right)  }\gamma^{-r}.
\]
This completes the proof of Lemma \ref{le2}.
\end{proof}

\subsection{Perron vectors of graphs with pendent paths}

In this subsection, we prove a structural inequality about the Perron vectors
of graphs with pendent paths.

\begin{lemma}
\label{le3}Let $\alpha\in\left[  0,1\right)  $ and $H$ be a connected graph
with $\rho_{\alpha}\left(  H\right)  \geq5/2$. Let $u$ and $v$ be vertices of
$H$. Attach a path $P:=\left\{  u_{1}=u,\ldots,u_{p}\right\}  $ to $u$ and a
path $Q:=\left\{  v_{1}=v,\ldots,v_{q}\right\}  $ to $v$, and write $G$ for
the resulting graph. Let $\mathbf{x}$ be a positive unit eigenvector to
$\rho:=\rho_{\alpha}\left(  G\right)  $; write $x_{1},\ldots,x_{p}$ for the
$\mathbf{x}$-entries of $u_{1},\ldots,u_{p}$ and $y_{1},\ldots,y_{q}$ for the
$\mathbf{x}$-entries of $v_{1},\ldots,v_{q}$. If $p\geq q+2$, then
\[
\frac{y_{q}}{x_{p-1}}\geq\frac{3y_{1}}{2x_{1}}.
\]

\end{lemma}

To prove Lemma \ref{le3}, we need a proposition:

\begin{proposition}
\label{pro2}Let $\alpha\in\left[  0,1\right)  $ and $\rho\geq5/2$. If $\gamma$
is defined as in (\ref{gam}), then%
\[
\gamma\geq\frac{2\rho-1-3\alpha}{2-2\alpha}.
\]

\end{proposition}

\begin{proof}
In view of (\ref{gam}), it is enough to prove that
\[
\sqrt{\left(  \rho-2\alpha\right)  ^{2}-4\left(  1-\alpha\right)  ^{2}%
}-\left(  \rho-1-\alpha\right)  \geq0.
\]
Indeed, we see that%
\begin{align*}
\sqrt{\left(  \rho-2\alpha\right)  ^{2}-4\left(  1-\alpha\right)  ^{2}%
}-\left(  \rho-1-\alpha\right)   &  =\frac{2\rho\left(  1-\alpha\right)
-\alpha^{2}+6\alpha-5}{\sqrt{\left(  \rho-2\alpha\right)  ^{2}-4\left(
1-\alpha\right)  ^{2}}+\left(  \rho-1-\alpha\right)  }\\
&  \geq\frac{5\left(  1-\alpha\right)  -\alpha+6\alpha-5}{\sqrt{\left(
\rho-2\alpha\right)  ^{2}-4\left(  1-\alpha\right)  ^{2}}+\left(
\rho-1-\alpha\right)  }\\
&  =0,
\end{align*}
completing the proof of Proposition \ref{pro2}.
\end{proof}

\bigskip

\begin{proof}
[\textbf{Proof of Lemma \ref{le3}}]To begin with, note that Proposition
\ref{pro1} implies that
\begin{equation}
x_{p-1}<\gamma^{-p+2}x_{1}, \label{bx}%
\end{equation}
and inequality (\ref{lb}) implies that
\begin{equation}
y_{q}>\frac{\left(  \gamma^{2}-1\right)  \left(  1-\alpha\right)  }%
{\gamma\left(  \left(  1-\alpha\right)  \gamma+\alpha\right)  }\gamma
^{-q+1}y_{1}. \label{by}%
\end{equation}
Dividing (\ref{by}) by (\ref{bx}), we see that%
\begin{equation}
\frac{y_{q}}{x_{p-1}}>\frac{\left(  \gamma^{2}-1\right)  \left(
1-\alpha\right)  }{\gamma\left(  \left(  1-\alpha\right)  \gamma
+\alpha\right)  }\gamma^{p-q-1}\frac{y_{1}}{x_{1}}\geq\frac{\left(  \gamma
^{2}-1\right)  \left(  1-\alpha\right)  }{\left(  1-\alpha\right)
\gamma+\alpha}\frac{y_{1}}{x_{1}}. \label{in1}%
\end{equation}
Thus, to finish the proof it is enough to show that
\[
\frac{\left(  \gamma^{2}-1\right)  \left(  1-\alpha\right)  }{\left(
1-\alpha\right)  \gamma+\alpha}\geq\frac{3}{2}.
\]
Assume for a contradiction that the above inequality fails, and recall that
$\gamma$ is a root of the equation (\ref{ceq}), implying in particular that
\[
\left(  \gamma^{2}-1\right)  \left(  1-\alpha\right)  =\left(  \rho
-2\alpha\right)  \gamma-2\left(  1-\alpha\right)  .
\]
Hence, we get
\[
\frac{3}{2}>\frac{\left(  \rho-2\alpha\right)  \gamma-2\left(  1-\alpha
\right)  }{\left(  1-\alpha\right)  \gamma+\alpha}\geq\frac{\left(
5/2-2\alpha\right)  \gamma-2+2\alpha}{\left(  1-\alpha\right)  \gamma+\alpha
},
\]
yielding in turn%
\[
\left(  3-3\alpha\right)  \gamma+3\alpha>\left(  5-4\alpha\right)
\gamma-4+4\alpha,
\]
and further%
\[
\gamma<\frac{4-\alpha}{2-\alpha}.
\]
To get the desired contradiction, we invoke Proposition \ref{pro2}, which
gives%
\[
\frac{4-\alpha}{2-\alpha}>\frac{4-3\alpha}{2-2\alpha}.
\]
and, after some algebra, this inequality reduces to
\[
3\alpha^{2}<2\alpha^{2},
\]
an apparent contradiction that completes the proof of Lemma \ref{le3}.
\end{proof}

\section{\label{smth}Proof of Theorem \ref{mth}}

Our proof of Theorem \ref{mth} is based on two independent results: first,
proving that if $G$ has a maximal $\alpha$-index among all graphs of order $n$
and diameter at least $k$, then $G$ is isomorphic to $B_{n-k+2,p,k-p}$ for
some $p$, satisfying $1\leq p\leq k$; and second, proving that among all
graphs $B_{n-k+2,p,k-p}$, $\left(  1\leq p\leq\left\lfloor k/2\right\rfloor
\right)  $, the maximal $\alpha$-index is attained when $p=\left\lfloor
k/2\right\rfloor $. We start with establishing the second result, which is
based on the following crucial lemma:

\begin{lemma}
\label{th1}Let $\alpha\in\left[  0,1\right)  $ and $k\geq4$. If $p\geq q+2$
and $q\geq1$, then%
\[
\rho_{\alpha}\left(  B_{k,p,q}\right)  <\rho_{\alpha}\left(  B_{k,p-1,q+1}%
\right)  .
\]

\end{lemma}

To prove the lemma, we need a lower bound on the $\alpha$-index of the graph
$K_{k}-e$:

\begin{proposition}
\label{pro3}If $k\geq4$, then%
\begin{equation}
\rho_{\alpha}\left(  K_{k}-e\right)  \geq k-3+2\alpha. \label{lob}%
\end{equation}

\end{proposition}

\begin{proof}
First, since the entries of an eigenvector to $K_{k}-e$ take only two values,
it is not hard to show that $\rho_{\alpha}\left(  K_{k}-e\right)  $ is the
larger root of the equation
\[
X^{2}-\left(  k-3+k\alpha\right)  X+\left(  k-2\right)  \left(  \left(
k+1\right)  \alpha-2\right)  =0.
\]
Hence, one gets%
\begin{align*}
\rho_{\alpha}\left(  K_{k}-e\right)   &  =\frac{1}{2}\left(  k-3+k\alpha
+\sqrt{\left(  k-3+k\alpha\right)  ^{2}-4\left(  k-2\right)  \left(  \left(
k+1\right)  \alpha-2\right)  }\right) \\
&  =\frac{1}{2}\left(  k-3+k\alpha+\sqrt{k^{2}\left(  1-\alpha\right)
^{2}+2\left(  k-4\right)  \left(  1-\alpha\right)  +1}\right)  .
\end{align*}
Assume for a contradiction that (\ref{lob}) fails, implying that
\[
k-3+k\alpha+\sqrt{k^{2}\left(  1-\alpha\right)  ^{2}+2\left(  k-4\right)
\left(  1-\alpha\right)  +1}<2k-6+4\alpha,
\]
and therefore,%
\[
\sqrt{k^{2}\left(  1-\alpha\right)  ^{2}+2\left(  k-4\right)  \left(
1-\alpha\right)  +1}<\left(  k-4\right)  \left(  1-\alpha\right)  +1.
\]
Squaring both sides, we get%
\[
k^{2}\left(  1-\alpha\right)  ^{2}+2\left(  k-4\right)  \left(  1-\alpha
\right)  +1<\left(  k-4\right)  \left(  1-\alpha\right)  ^{2}+2\left(
k-4\right)  \left(  1-\alpha\right)  +1,
\]
which is an obvious contradiction. Proposition \ref{pro3} is proved.
\end{proof}

\bigskip

\begin{proof}
[\textbf{Proof of Lemma \ref{th1}}]We adopt the setup of Lemma \ref{le3}: Let
$u$ and $v$ be the two nonadjacent vertices of $B_{k,p,q}$; let $P:=\left\{
u_{1}=u,\ldots,u_{p}\right\}  $ be the path attached to $u$ and a $Q:=\left\{
v_{1}=v,\ldots,v_{q}\right\}  $ be the path attached to $v$. Let $\mathbf{x}$
be a positive unit eigenvector to $\rho:=\rho_{\alpha}\left(  B_{k,p,q}%
\right)  $; write $x_{1},\ldots,x_{p}$ for the $\mathbf{x}$-entries of
$u_{1},\ldots,u_{p}$ and $y_{1},\ldots,y_{q}$ for the $\mathbf{x}$-entries of
$v_{1},\ldots,v_{q}$. For convenience, write $H$ for the subgraph of
$B_{k,p,q}$ that is isomorphic to $K_{k}-e$.

Note that by deleting the edge $\left\{  u_{p},u_{p-1}\right\}  $ and adding
the edge $\left\{  u_{p},v_{q}\right\}  $, the graph $B_{k,p,q}$ is
transformed into $B_{k,p-1,q+1}$. Now, calculating the quadratic forms
$\left\langle A_{\alpha}\left(  B_{k,p-1,q+1}\right)  \mathbf{x}%
,\mathbf{x}\right\rangle $ and $\left\langle A_{\alpha}\left(  B_{k,p,q}%
\right)  \mathbf{x},\mathbf{x}\right\rangle $, we see that%
\begin{align*}
\rho_{\alpha}\left(  B_{k,p-1,q+1}\right)  -\rho_{\alpha}\left(
B_{k,p,q}\right)   &  \geq\left\langle A_{\alpha}\left(  B_{k,p-1,q+1}\right)
\mathbf{x},\mathbf{x}\right\rangle -\left\langle A_{\alpha}\left(
B_{k,p,q}\right)  \mathbf{x},\mathbf{x}\right\rangle \\
&  =-\alpha x_{p-1}^{2}+\alpha y_{q}^{2}+2\left(  1-\alpha\right)
x_{p}\left(  -x_{p-1}+y_{q}\right) \\
&  =\left(  \alpha x_{p-1}+\alpha y_{q}+2\left(  1-\alpha\right)
x_{p}\right)  \left(  y_{q}-x_{p-1}\right)  .
\end{align*}
To prove the theorem, it is enough to show that $y_{q}/x_{p-1}\geq1$, thus,
this inequality is our goal to the end of the proof.

First, Lemma \ref{le3} gives
\begin{equation}
\frac{y_{q}}{x_{p-1}}>\frac{3y_{1}}{2x_{1}}, \label{in}%
\end{equation}
thereby our task is reduced to showing that $y_{1}/x_{1}\geq2/3$. Further, the
eigenequations of $A_{\alpha}\left(  G\right)  $ for $u$ and $v$ give
\begin{align*}
\rho x_{1}  &  =\alpha\left(  k-1\right)  x_{1}+\left(  1-\alpha\right)
\sum_{i\in\Gamma_{H}\left(  u\right)  \backslash\left\{  v\right\}  }%
x_{i}+\left(  1-\alpha\right)  x_{2},\\
\rho y_{1}  &  \geq\alpha\left(  k-1\right)  y_{1}+\left(  1-\alpha\right)
\sum_{i\in\Gamma_{H}\left(  v\right)  \backslash\left\{  u\right\}  }x_{i}.
\end{align*}
Since $\sum_{i\in\Gamma_{H}\left(  u\right)  \backslash\left\{  v\right\}
}x_{i}=\sum_{i\in\Gamma_{H}\left(  v\right)  \backslash\left\{  u\right\}
}x_{i}$, writing $S$ for $\sum_{i\in\Gamma_{H}\left(  u\right)  \backslash
\left\{  v\right\}  }x_{i}$, we see that
\begin{align*}
\left(  \rho-\left(  k-1\right)  \alpha\right)  x_{1}  &  =\left(
1-\alpha\right)  S+\left(  1-\alpha\right)  x_{2}\\
&  <\left(  1-\alpha\right)  S+\gamma^{-1}\left(  1-\alpha\right)  x_{1},
\end{align*}
yielding in turn%
\[
\left(  \rho-\left(  k-1\right)  \alpha-\gamma^{-1}\left(  1-\alpha\right)
\right)  x_{1}<\left(  1-\alpha\right)  S.
\]
Likewise, we get
\[
\left(  \rho-\left(  k-1\right)  \alpha\right)  y_{1}\geq\left(
1-\alpha\right)  S.
\]
Dividing the last inequality by the previous one, we find that
\begin{equation}
\frac{y_{1}}{x_{1}}>\frac{\rho-\left(  k-1\right)  \alpha-\gamma^{-1}\left(
1-\alpha\right)  }{\rho-\left(  k-1\right)  \alpha}=1-\frac{1-\alpha}{\left(
\rho-\left(  k-1\right)  \alpha\right)  \gamma}. \label{in2}%
\end{equation}
Hence, to show that $y_{1}/x_{1}\geq2/3$, it is enough to prove the
inequality
\[
\frac{1-\alpha}{\left(  \rho-\left(  k-1\right)  \alpha\right)  \gamma}%
\leq\frac{1}{3}.
\]
Proposition \ref{pro3} gives%
\[
\frac{1-\alpha}{\left(  \rho-\left(  k-1\right)  \alpha\right)  \gamma}%
<\frac{1-\alpha}{\left(  k-3+2\alpha-\left(  k-1\right)  \alpha\right)
\gamma}=\frac{1}{\left(  k-3\right)  \gamma}.
\]
In particular, since $\gamma\geq2$, the theorem is proved for $k\geq5$.
Moreover, if $k=4$, Proposition \ref{pro2} implies that $\gamma\geq3$, as long
as $\alpha\geq2/3$; hence, the theorem is proved also for $k=4$ and
$\alpha\geq2/3$. In the remaining case ($k=4$ and $\alpha<2/3$), using the
bounds $\rho\geq5/2$ and $\gamma\geq2$, we see that
\[
\frac{1-\alpha}{\left(  \rho-\left(  k-1\right)  \alpha\right)  \gamma}%
\leq\frac{1-\alpha}{\left(  5/2-3\alpha\right)  2}=\frac{1-\alpha}{5-6\alpha
}<\frac{1}{3}.
\]
Theorem \ref{th1} is proved.
\end{proof}

\begin{corollary}
\label{cor1}If $p$ and $q+r$ are fixed, then
\[
\rho_{\alpha}(B_{p,q,r})\leq\rho_{\alpha}(B_{p,\left\lfloor \left(
q+r\right)  /2\right\rfloor ,\left\lceil \left(  q+r\right)  /2\right\rceil
})
\]

\end{corollary}

\bigskip

\begin{proof}
[\textbf{Proof of Theorem \ref{mth}}]The statement is clear if $k=1$, for
$K_{n}$ is the only graph of order $n$ and diameter $1$. Suppose that $k\geq
2$, and let $G$ be a graph with maximal $\alpha$-index among all graphs of
order $n$ and $\mathrm{diam}\left(  G\right)  \geq k$. This choice implies
that $G$ is edge-maximal, that is, no edge can be added to $G$ without
diminishing its diameter; in particular, $\mathrm{diam}\left(  G\right)  =k$.
In the light of Corollary \ref{cor1}, we only need to show that $G=$
$B_{n-k+2,p,k-p}$ for some $p$, satisfying $1\leq p\leq k-1$.

Set for short $\rho:=\rho_{\alpha}\left(  G\right)  $. Let $u,v$ be vertices
of $G$ at distance exactly $k$, and for every $i=0,\ldots,k$, let $V_{i}$ be
the set of vertices at distance $i$ from $u$. Since $G$ is edge-maximal, for
every $i=0,\ldots,k-1$, the set $V_{i}\cup V_{i+1}$ induces a complete graph.
It is also clear that $\left\vert V_{0}\right\vert =1$; moreover, it is not
hard to see that $\left\vert V_{k}\right\vert =1$. Indeed, assume for a
contradiction that $\left\vert V_{k}\right\vert \geq2$, and add all edges
between $V_{k-2}$ and $V\backslash\left\{  v\right\}  $. These additional
edges do not diminish the distance between $u$ and $v$; hence $G$ is not
edge-maximal, contradicting its choice; therefore $\left\vert V_{k}\right\vert
=1$.

Further, well-known bounds for $\rho_{\alpha}\left(  G\right)  $ show that
\[
\Delta\left(  G\right)  \geq\rho_{\alpha}\left(  G\right)  >\rho_{\alpha
}(B_{n-k+2,\lfloor k/2\rfloor,\lceil k/2\rceil})>\rho_{\alpha}(K_{n-k+2}%
-e)>\delta\left(  K_{n-k+2}-e\right)  =n-k,
\]
and so, $\Delta\left(  G\right)  \geq n-k+1.$ Suppose that $w$ is vertex of
maximum degree in $G$, and let say $w\in V_{i}$. Clearly, $0<i<k$, and in view
of
\[
d\left(  w\right)  =\left\vert V_{i-1}\right\vert +\left\vert V_{i}\right\vert
+\left\vert V_{i+1}\right\vert -1,
\]
we find that
\[
n-k+2\leq\left\vert V_{i-1}\right\vert +\left\vert V_{i}\right\vert
+\left\vert V_{i+1}\right\vert =n-\sum_{j<i-1}\left\vert V_{j}\right\vert
-\sum_{j>i+1}\left\vert V_{j}\right\vert =n-\left(  k+1-3\right)  =n-k+2.
\]
Hence, if $j<i-1$ or $j>i+1$, then $\left\vert V_{j}\right\vert =1$;
furthermore, $\left\vert V_{i-1}\right\vert +\left\vert V_{i}\right\vert
+\left\vert V_{i+1}\right\vert =n-k+2$.

If $\left\vert V_{i-1}\right\vert =\left\vert V_{i+1}\right\vert =1$, then
obviously $G=$ $B_{n-k+2,i,k-i}$, so Theorem \ref{mth} is proved in this case.
We shall show that all other cases lead to contradictions, by constructing a
graph $H$ of order $n$ and $\mathrm{diam}\left(  G\right)  =k$ with
$\rho_{\alpha}\left(  H\right)  >\rho$. Suppose that $\mathbf{x}:=\left(
x_{1},\ldots,x_{n}\right)  $ is a positive unit vector to $\rho_{\alpha
}\left(  G\right)  $.

First, consider the case $\left\vert V_{i-1}\right\vert =1$ and $\left\vert
V_{i+1}\right\vert \geq2$. If $\left\vert V_{i}\right\vert =1$, the proof is
completed, so we suppose that $\left\vert V_{i}\right\vert \geq2$. Let
$V_{i-1}=\left\{  a\right\}  ,$ $V_{i+2}=\left\{  b\right\}  ,$ and suppose by
symmetry that $x_{b}\geq x_{a}$. Choose a vertex $w\in V_{i}$, delete the edge
$\left\{  w,a\right\}  $, add the edge $\left\{  w,b\right\}  $, and write $H$
for the resulting graph. In other words, $H$ is obtained by moving the vertex
$w$ from $V_{i}$ into $V_{i+1}$. By symmetry, $x_{w^{\prime}}=x_{w}$ for any
$w^{\prime}\in V_{i}$; thus the choice of $G$ implies that
\begin{align*}
0  &  \geq\rho_{\alpha}\left(  H\right)  -\rho\geq\left\langle A_{\alpha
}\left(  H\right)  \mathbf{x},\mathbf{x}\right\rangle -\left\langle A_{\alpha
}\left(  G\right)  \mathbf{x},\mathbf{x}\right\rangle \\
&  \geq\left(  1-\alpha\right)  \left(  \left\langle A\left(  H\right)
\mathbf{x},\mathbf{x}\right\rangle -\left\langle A\left(  G\right)
\mathbf{x},\mathbf{x}\right\rangle \right)  +\alpha\left(  \left\langle
D\left(  H\right)  \mathbf{x},\mathbf{x}\right\rangle -\left\langle D\left(
G\right)  \mathbf{x},\mathbf{x}\right\rangle \right)  .
\end{align*}
On the other hand, it is not hard to see that%
\begin{align*}
\left\langle A\left(  H\right)  \mathbf{x},\mathbf{x}\right\rangle
-\left\langle A\left(  G\right)  \mathbf{x},\mathbf{x}\right\rangle  &
=2x_{w}\left(  x_{b}-x_{a}\right)  ,\\
\left\langle D\left(  H\right)  \mathbf{x},\mathbf{x}\right\rangle
-\left\langle D\left(  G\right)  \mathbf{x},\mathbf{x}\right\rangle  &
=x_{b}^{2}-x_{a}^{2}.
\end{align*}
Hence,%
\[
0\geq\left(  x_{b}-x_{a}\right)  \left(  2\left(  1-\alpha\right)
x_{w}+\alpha\left(  x_{b}+x_{a}\right)  \right)  \geq0,
\]
implying that $\rho_{\alpha}\left(  H\right)  =\rho_{\alpha}\left(  G\right)
$ and that $\mathbf{x}$ is an eigenvector to $\rho_{\alpha}\left(  H\right)
$. However, the neighborhood of $a$ in $H$ is a proper subset of the
neighborhood of $a$ in $G$, so the eigenequations for $\rho_{\alpha}\left(
H\right)  $ and $\rho_{\alpha}\left(  G\right)  $ for the vertex $a$ are contradictory.

The same argument disposes also of the case $\left\vert V_{i-1}\right\vert
\geq2$ and $\left\vert V_{i+1}\right\vert =1$; thus, to complete the proof, it
remains to consider the case $\left\vert V_{i-1}\right\vert \geq2$ and
$\left\vert V_{i+1}\right\vert \geq2$.

Let $V_{i-2}=\left\{  a\right\}  ,$ $c\in V_{i-1},$ $d\in V_{i+1},$ and
$V_{i+2}=\left\{  b\right\}  $. Our first step is to show that
\begin{equation}
x_{c}>x_{a}. \label{inx}%
\end{equation}
Note that if $i\geq3$, and $V_{i-3}=\left\{  z\right\}  $, then Proposition
\ref{pro1} gives $x_{z}<x_{a}$. Hence, setting $l:=\left\vert V_{i-1}%
\right\vert $, the eigenequation for the vertex $a$ implies that
\[
\rho x_{a}<\alpha\left(  l+1\right)  x_{a}+\left(  1-\alpha\right)  x_{\alpha
}+\left(  1-\alpha\right)  lx_{c},
\]
yielding in turn%
\[
x_{c}>\frac{\rho-1-\alpha l}{1-\alpha}x_{\alpha}=\left(  \rho-1+\frac
{\rho-1-l}{1-\alpha}\alpha\right)  x_{a}.
\]
Since
\[
\rho-1\geq n-k-1=\left\vert V_{i-1}\right\vert +\left\vert V_{i}\right\vert
+\left\vert V_{i+1}\right\vert -3\geq\left\vert V_{i-1}\right\vert =l,
\]
inequality (\ref{inx}) is proved. By symmetry, we also see that $x_{d}>x_{b}$.
Suppose, again by symmetry, that $x_{b}\geq x_{a}$, which yields $x_{a}\leq
x_{b}<x_{d}$. Choose a vertex $w\in V_{i-1}$, delete the edge $\left\{
w,a\right\}  $, add the edges $\left\{  w,s\right\}  $ for all $s\in V_{i+1}$,
and write $H$ for the resulting graph. In other words, $H$ is obtained by
moving the vertex $w$ from $V_{i-1}$ into $V_{i}$. The choice of $G$ implies
that
\begin{align*}
0  &  \geq\rho_{\alpha}\left(  H\right)  -\rho\geq\left\langle A_{\alpha
}\left(  H\right)  \mathbf{x},\mathbf{x}\right\rangle -\left\langle A_{\alpha
}\left(  G\right)  \mathbf{x},\mathbf{x}\right\rangle \\
&  \geq\left(  1-\alpha\right)  \left(  \left\langle A\left(  H\right)
\mathbf{x},\mathbf{x}\right\rangle -\left\langle A\left(  G\right)
\mathbf{x},\mathbf{x}\right\rangle \right)  +\alpha\left(  \left\langle
D\left(  H\right)  \mathbf{x},\mathbf{x}\right\rangle -\left\langle D\left(
G\right)  \mathbf{x},\mathbf{x}\right\rangle \right)  .
\end{align*}
On the other hand, it is not hard to see that
\begin{align*}
\left\langle A\left(  H\right)  \mathbf{x},\mathbf{x}\right\rangle
-\left\langle A\left(  G\right)  \mathbf{x},\mathbf{x}\right\rangle  &
=2x_{w}\left\vert V_{i+1}\right\vert x_{d}-2x_{w}x_{a}>\left(  x_{d}%
-x_{a}\right)  x_{w}\\
\left\langle D\left(  H\right)  \mathbf{x},\mathbf{x}\right\rangle
-\left\langle D\left(  G\right)  \mathbf{x},\mathbf{x}\right\rangle  &
=\left\vert V_{i+1}\right\vert x_{d}^{2}-x_{a}^{2}>x_{d}^{2}-x_{a}^{2}.
\end{align*}
Hence,%
\[
0\geq2\left(  1-\alpha\right)  \left(  x_{d}-x_{a}\right)  x_{w}+\alpha\left(
x_{d}^{2}-x_{a}^{2}\right)  =\left(  x_{d}-x_{a}\right)  \left(  2\left(
1-\alpha\right)  x_{w}+\alpha\left(  x_{d}+x_{a}\right)  \right)  >0.
\]
This contradiction completes the proof of Theorem \ref{mth}.
\end{proof}

\section{\label{sthcl}Proof of Theorem \ref{thcl}}

Our\ proof of Theorem can be broken into several distinct steps, which are
formulated below as separate propositions in a slightly more general form.

\begin{proposition}
\label{pro4}Let $\alpha\in\left[  0,1\right)  $, let $G$ be a connected graph
with $\rho_{\alpha}\left(  G\right)  \geq2$, and let $u\in V\left(  G\right)
$. If $q\geq1$ and $p\geq q$, then
\[
\rho_{\alpha}\left(  G_{p,q}\left(  u\right)  \right)  \geq\rho_{\alpha
}\left(  G_{p+q-1,1}\left(  u\right)  \right)
\]
with equality if and only if $q=1.$
\end{proposition}

\begin{proof}
If $q=1$, there is nothing to prove, so suppose that $q\geq2$. Let
$P_{p+q-1}=\left(  v_{1}=u,\ldots,v_{p+q-1}\right)  $ be the path attached to
$u$, let $\mathbf{x}$ be a positive eigenvector to $\rho=\rho_{\alpha}\left(
G_{p+q-1,1}\left(  u\right)  \right)  $, and let $x_{1},\ldots,x_{p+q-1}$ be
the $\mathbf{x}$-entries of $v_{1},\ldots,v_{p+q-1}$. Delete the edge
$\left\{  v_{p+1},v_{p}\right\}  $ and add the edge $\left\{  v_{p+1}%
,v_{1}\right\}  $, thus obtaining the graph $G_{p,q}\left(  u\right)  $. Since
Corollary \ref{cor0} implies that $x_{1}>x_{p}$, Lemma \ref{rol} implies that
$\rho_{\alpha}\left(  G_{p,q}\left(  u\right)  \right)  >\rho_{\alpha}\left(
G_{p+q-1,1}\left(  u\right)  \right)  $.
\end{proof}

\begin{proposition}
\label{pro5}Let $\alpha\in\left[  0,1\right)  $, let $G$ be a connected graph
with $\rho_{\alpha}\left(  G\right)  \geq2$, and let $u\in V\left(  G\right)
$. It $G_{u}\left(  T\right)  $ is the graph obtained by identifying $u$ with
a vertex of a tree $T$ of order $n$, then
\[
\rho_{\alpha}\left(  G_{T}\left(  u\right)  \right)  \geq\rho_{\alpha}\left(
G_{n,1}\left(  u\right)  \right)  .
\]
with equality if and only if $G_{T}\left(  u\right)  =G_{n,1}\left(  u\right)
$.
\end{proposition}

We omit the proof of Proposition \ref{pro5}, which can be carried out along
well-known lines, by applying Proposition \ref{pro4} to recursively flatten
$T$ until it becomes a path (see, e.g., \cite{StHa08} for more
details.)\bigskip

\begin{proof}
[\textbf{Proof of Theorem \ref{thcl}}]Let $G$ be a graph with minimal $\alpha
$-index among all connected graphs of order $n$ and clique number $\omega$. If
$\omega=2$, then $G$ must be a path, as the path it the graph with smallest
$\alpha$-index among connected graphs of given order (see \cite{NPRS17}).
Thus, we suppose that $\omega\geq3$ and let $H$ be a complete subgraph of $G$
of order $\omega$.

Further, $G$ should be edge-minimal, that is, the removal of any edge of $G$
either makes $G$ disconnected or its clique number diminishes. In particular,
if $G^{\prime}$ is the graph obtained by removing the edges of $H$, then the
components of $G^{\prime}$ are trees, and each component has exactly one
vertex in common with $H$. It follows that $G$ is isomorphic to a complete
graph of order $\omega$ with trees attached to some of its vertices. Moreover,
Proposition \ref{pro5} implies that each of those trees must be a pendent
path. To complete the proof, we show that there is only one such path.

Let $S=V\left(  H\right)  $, let $u,v\in S$, and suppose that a path
$P_{p}=\left(  v_{1}=v,\ldots,v_{p}\right)  $ is attached to $v$ and
$P_{q}=\left(  u_{1}=u,\ldots,u_{q}\right)  $ is attached to $v$. Let $F$ be
the graph obtained by deleting the edge $\left\{  u_{2},u_{1}\right\}  $ and
adding the edge $\left\{  u_{2},v_{p}\right\}  $, that is, $F$ is obtained by
removing $P_{q}$ and extending $P_{p}$ to $P_{p+q-1}$. To complete the proof,
we need to show that $\rho_{\alpha}\left(  G\right)  >\rho_{\alpha}\left(
F\right)  $.

Let $\rho=\rho_{\alpha}\left(  F\right)  $ and let $\mathbf{x}$ be a positive
eigenvector of $F$ to $\rho$. Write $x_{1},\ldots,x_{p+q-1}$ for the entries
of $\mathbf{x}$ corresponding to $v_{1},\ldots v_{p},u_{2},\ldots,u_{q}$, and
let $\gamma$ be defined by (\ref{gam}). Now, if $x_{u}\geq\gamma^{-1}x_{v}$,
then Proposition \ref{pro1} implies that $x_{u}\geq\gamma^{-1}x_{1}>x_{2}\geq
x_{p}$, and so, Lemma \ref{rol} implies that $\rho_{\alpha}\left(  G\right)
>\rho_{\alpha}\left(  F\right)  $. Thus, we focus on showing that $x_{u}%
\geq\gamma^{-1}x_{1}$.

On the one hand, the eigenequation for $u$ is
\[
\rho x_{u}=\left(  \omega-1\right)  \alpha x_{u}+\left(  1-\alpha\right)
\sum_{i\in S\backslash\left\{  u\right\}  }x_{i}=\left(  \omega\alpha
-1\right)  x_{u}+\left(  1-\alpha\right)  \sum_{i\in S}x_{i},
\]
and therefore,
\[
\left(  \rho+1-\omega\alpha\right)  x_{u}=\left(  1-\alpha\right)  \sum_{i\in
S}x_{i}.
\]
Likewise, for any $w\in S\backslash\left\{  u,v\right\}  $, the eigenequation
for $w$ gives
\[
\left(  \rho+1-\omega\alpha\right)  x_{w}\geq\left(  1-\alpha\right)
\sum_{i\in S}x_{i}.
\]
In particular, we see that $x_{w}\geq x_{u}$ for any $w\in S\backslash\left\{
u,v\right\}  $, since $\rho+1-\omega\alpha>\omega\left(  1-\alpha\right)  >0$.

Returning to the eigenequation for $u$, we find that%
\[
\rho x_{u}=\left(  \omega-1\right)  \alpha x_{u}+\left(  1-\alpha\right)
\sum_{i\in S\backslash\left\{  u\right\}  }x_{i}\geq\left(  \omega-1\right)
\alpha x_{u}+\left(  1-\alpha\right)  \left(  x_{1}+\left(  \omega-2\right)
x_{u}\right)  ,
\]
yielding in turn%
\[
\left(  \rho-\omega+2-\alpha\right)  x_{u}\geq\left(  1-\alpha\right)  x_{1}.
\]
Assuming for a contradiction that $x_{u}<\gamma^{-1}x_{1}$, after some
algebra, we get
\[
\frac{\rho-\omega+2-\alpha}{1-\alpha}>\gamma=\frac{\rho-2\alpha}{2\left(
1-\alpha\right)  }+\frac{1}{2\left(  1-\alpha\right)  }\sqrt{\left(
\rho-2\alpha\right)  ^{2}-4\left(  1-\alpha\right)  ^{2}}%
\]
and therefore,%
\begin{equation}
\rho-2\omega+4>\sqrt{\left(  \rho-2\alpha\right)  ^{2}-4\left(  1-\alpha
\right)  ^{2}}. \label{in3}%
\end{equation}
It is known that $\rho<\Delta\left(  F\right)  =\omega$, since $F$ is not a
$\omega$-regular graph. Hence,
\[
\omega-2\omega+4>0,
\]
a contradiction if $\omega\geq4.$ If $\omega=3,$ then (\ref{in3}) becomes%
\[
\rho-2>\sqrt{\left(  \rho-2\alpha\right)  ^{2}-4\left(  1-\alpha\right)  ^{2}%
}.
\]
Squaring both sides of this inequality, we get%
\[
\rho^{2}-4\rho+4>\rho^{2}-4\alpha\rho-4+8\alpha
\]
and so,%
\[
8\left(  1-\alpha\right)  >4\rho\left(  1-\alpha\right)  .
\]
Therefore $\rho<2$, an obvious contradiction, completing the proof of Theorem
\ref{thcl}.
\end{proof}

\section{\label{ops}Some open problems}

In this section we raise a few problems inspired by the results of Li and Feng
\cite{LiFe79}, which have been presented in some detail in Section 6.2 of
\cite{CRS97}, and in Section 8.1 of \cite{CRS10}.

\begin{conjecture}
\label{con1}Let $\alpha\in\left[  0,1\right)  $. If $G$ is a connected graph
and $p\geq q+2\geq3$, then%
\[
\rho_{\alpha}\left(  G_{p,q}\left(  u\right)  \right)  <\rho_{\alpha}\left(
G_{p-1,q+1}\left(  u\right)  \right)  .
\]

\end{conjecture}

As shown by Li and Feng in \cite{LiFe79}, the above statement is true for
$\alpha=0$ (see also Theorem 8.1.20 in \cite{CRS10}). Moreover, Cvetkovi\'{c}
and Simi\'{c} \cite{CvSi09} showed that the statement is true for $\alpha=1/2$
as well. However, none of these techniques applies directly for other
$\alpha\in\left[  0,1\right)  $. Using edge rotation, we can show that the
statement is true for any $\alpha\in\left[  0,1\right)  $ as long as
$\rho_{\alpha}\left(  G_{p,q}\left(  u\right)  \right)  \geq9/4$, but this
constraint seems unnecessary strong.

Similar questions can be studied for pendent paths attached to different
vertices of a connected graph $G$.

\begin{conjecture}
\label{con2}Let $\alpha\in\left[  0,1\right)  $. Let $G$ be a connected graph,
and let $u$ and $v$ be adjacent vertices of $G$ of degree at least $2$. If
$q\geq1$ and $p\geq q+2$, then%
\begin{equation}
\rho_{\alpha}\left(  G_{p,q}\left(  u,v\right)  \right)  <\rho_{\alpha}\left(
G_{p-1,q+1}\left(  u,v\right)  \right)  . \label{cin1}%
\end{equation}

\end{conjecture}

It has been shown by Li and Feng \cite{LiFe79} that the above statement is
true for $\alpha=0$ (see also Theorem 8.1.22 in \cite{CRS10}). Again, their
methods seem not immediately applicable to Conjecture \ref{con2}.

Note that the requirement that the degree of $u$ and $v$ be at least $2$ is
important, for otherwise strict inequality may not always hold in
(\ref{cin1}); e.g., if $G$ is an edge and $u,v$ are its ends, then
$G_{3,1}\left(  u,v\right)  =G_{2,2}\left(  u,v\right)  =P_{4}$. The
requirement $d_{G}\left(  u\right)  \geq2$ and $d_{G}\left(  v\right)  \geq2$
has been omitted in Lemma 2.1 of \cite{StHa08} and in Lemma 4.3 of
\cite{Ste14}, which makes these statements technically incorrect, although
their applications in \cite{StHa08} and in \cite{Ste14} are fine.

Further, Lemma \ref{th1} suggests that the requirement for $u$ and $v$ to be
adjacent may not always be necessary, so we raise the following question:

\begin{question}
\label{que1}For which connected graphs $G$ the following statement is true:

Let $\alpha\in\left[  0,1\right)  $ and let $u$ and $v$ be non-adjacent
vertices of $G$ of degree at least $2$. If $q\geq1$ and $p\geq q+2$, then%
\[
\rho_{\alpha}\left(  G_{p,q}\left(  u,v\right)  \right)  <\rho_{\alpha}\left(
G_{p-1,q+1}\left(  u,v\right)  \right)  .
\]

\end{question}

Little seems known about Question \ref{que1}, even for $\alpha=0$. It is not
hard to find examples of trees showing that the opposite inequality may hold sometimes.

\bigskip

\textbf{Acknowledgement }The research of O. Rojo was supported by
\emph{Project Fondecyt Regular 1170313, Chile}.

\end{document}